\documentclass{article}
\usepackage{inputenc}
\usepackage{amscd}
\usepackage{amsmath}
\usepackage{bm}
\usepackage{amssymb}
\usepackage{breqn}
\usepackage{pstricks}
\usepackage{wasysym}

\newtheorem{lem}{Lemma}
\newtheorem{prop}{Proposition}
\newtheorem{rem}{Remark}
\newtheorem{thm}{Theorem}

\newenvironment{proof}{{\noindent\it{Proof}}.}{\hfill$\Box$}
\newcommand{\R}{\mathbb{R}}

\newcommand\ddfrac[2]{\frac{\displaystyle #1}{\displaystyle #2}}

\title{On persistence of a Nicholson-type system with multiple delays and nonlinear harvesting}

\author{
{Pablo Amster}
\thanks{E-mail: \texttt{pamster@dm.uba.ar}}\\%
{Melanie Bondorevsky}%
\thanks{E-mail: \texttt{mbondo@dm.uba.ar}}\\%
Departamento de Matem\'atica - FCEyN\\ Universidad de Buenos Aires \& IMAS-CONICET\\ 
}
\date{}

\begin{document}

\maketitle


\begin{abstract}

An $N$-dimensional generalization of  Nicholson's equation is analyzed. We consider a model including multiple delays, nonlinear coefficients and a nonlinear harvesting term.
{Inspired by previous results in this subject,} 
 we obtain sufficient conditions to guarantee strong and uniform persistence. 
Furthermore, under extra suitable hypotheses we prove the existence of $T$-periodic solutions and, reversing the prior conditions in a convenient manner, we  show that the zero is a global attractor.
\end{abstract}




\section{Introduction}


The Nicholson equation was originally proposed in 1980 by Gurney, Blythe and Nisbet for a population model, 
though it was actually named after the entomologist A.J. Nicholson \cite{N}, who studied a family population of Lucilia cuprina flies in 1954. 
The classical equation is defined as follows
\begin{equation}\label{nich}
x'(t)=-d x(t) + p x(t-\tau)e^{-ax(t-\tau)} \qquad \hbox{ $t\geq 0$,}
\end{equation} 

\noindent where the function $x(t)$ describes the population density,
the constant $d>0$ represents the individual mortality rate, 
$p>0$ stands for the highest rate of egg production, 
$a^{-1}>0$ it is the value at which the function $g(x)=xe^{-ax}$ shifts its monotonicity and $\tau>0$ refers to the expected duration of the maturation process. 
For a detailed model deduction see \cite{Gurney} and for a wide discussion see \cite{BBI}.  

\medskip



In this work we study an $N$-dimensional generalization of  Nicholson's equation,
namely

\begin{equation}\label{nich-sys-gral}
x_i'(t)=-d_i(t, x_i(t)) + \sum_{l=1, l\neq i}^{N} b_{i,l}(t, x_l(t))+ \sum_{j=1}^{k} p_{i,j}(t) f(x_{i,\tau_j}(t)) -H_i(t,x(t))
\end{equation}

\noindent where the vector $x=(x_1,...,x_N)$ represents the $N$ different populations over time, 
$d_i(t, x_i)$ are the mortality rates, 
$b_{i,l}(t, x_i)$ are mutualism terms, which will be interpreted as the collaborations between $x_l$ and $x_i$,
and $p_{i,j}(t)$ are functions of production rate for each delay $\tau_j>0$.
Throughout the paper, we shall denote by $x_{i,\tau_j}(t)$ the delayed coordinate $x_{i}(t-\tau_j)$ and set $f(y):=ye^{-y}$.
We shall consider, for each $i=1,\ldots, N$, state-dependent  harvesting terms $H_i(t, x_1,..., x_N)$.
In our setting, an initial condition for the system (\ref{nich-sys-gral}) shall be given by a continuous function
$\varphi:[-\tau^{*},0]\to[0,+\infty)^N$, where $\tau^{*}=\max\{\tau_j\}$. 

\medskip



We shall focus our attention, on the one hand, on the generalization of the original scalar  
equation to $N$-dimensions and, on the other hand, on the effect of the nonlinear harvesting terms. 
For a more comprehensive presentation, we shall begin our analysis with the scalar case, which will guide the study of the $N$-dimensional system.
More specifically, for the general case we shall use guiding function techniques (see \cite{AB}) that will allow us to reduce the approach to a one-dimensional problem. 
It is worth mentioning that 
the treatment is different  with respect to the scalar equation, due to the presence of the mutualism terms, which do not allow the application of  uncoupling techniques.

In order to analyze the first case, we shall take into account  previous results, such as the work of Berezansky, Braverman and Idels \cite{BBI} and extend it to our framework.
In more precise terms, we will guarantee strong and uniform persistence by imposing conditions over an auxiliary function $G^0$,  that shall be defined below.  
If we reduce the problem by considering linear terms only, the conditions on $G^0$ can be translated into those described by the authors.
Here, we shall assume instead that the functions are asymptotically linear and, using this fact, accurate assumptions 
will allow us to obtain  the results. 
In addition, from the persistence analysis and, considering upper bounds for the solutions, we shall prove the existence of $T$-periodic solutions 
by means of degree theory and 
an accurate adaptation of the Poincaré-Miranda theorem.
Finally, we shall show that if  the conditions for persistence are reversed in a particular way, then zero is a global attractor for the positive trajectories.  

\medskip



It is worth recalling that Nicholson's equation and extensions
have been studied by different  authors. 
Earlier work includes: \cite{AD2}, \cite{AD}, \cite{BB}, \cite{CH}, \cite{FA}, \cite{FA2}, \cite{LD}, \cite{LG}, \cite{OS}, \cite{ST} and references therein. 
A number of cases involving multiple delays, nonlinear coefficients and linear harvesting terms have been also explored.


For instance, 
Obaya and Sanz \cite{OS} studied the persistence of an almost periodic form of a general Nicholson type equation:  
\begin{equation}\label{obaya}
X'(t)=-M(t) X(t) +D(t) X(t)+B(t,X(t-\tau))\quad\text{if $t\geq 0$.}
\end{equation}
 
\noindent In this framework, $M(t)$ is an $N\times N$ diagonal matrix of continuous functions corresponding to the time–varying mortality rate, 
$D(t)$ is an $N\times N$ matrix of continuous functions corresponding to the time–varying dispersion of the model, 
and the function $B:[0,+\infty)\times\R^N\to\R^N$ is the birth rate.  
The authors establish sufficient conditions for uniform and strict persistence in the case of skew-product semi-flows generated by solutions of 
families of cooperative DDE systems with time-recurring behavior, in terms of the principal spectra of some associated linearized skew-product semi-flows, which admit continuous separation. 
Their results are applied to an almost periodic version of (\ref{obaya}), which is a locally cooperative system, and show that its persistence is equivalent to the persistence of a linearization of the system around the null solution.


Also, we highlight the work of Faria, Obaya and Sanz \cite{Faria18}, that extends the above results to the nonautonomous framework, exploiting the properties of the cooperative linear ODEs  system $X'(t)=[M(t)-D(t)]X(t)$, and assuming conditions to guarantee the exponential globally stability. 
Finally, comparison techniques are applied to an auxiliary cooperative system in order to obtain, among other results, sufficient conditions for the persistence of the system. 


In \cite{FA},  Faria studied several population models with continuous delays and showed the existence of at least one periodic solution via Schauder's fixed point theorem. The proof relies on two crucial dynamical properties of the system: dissipativeness and uniform persistence.


In a recent work from Ossandon and Sepúlveda \cite{OSE}, 
they prove the existence of positive solutions of periodic variants of (\ref{obaya}) by the means of topological methods.  

Additional results for  Nicholson-type equations were analyzed by the first author, together with Balderrama and Idels \cite{ABI}  and D\'eboli \cite{AD}. 

In our previous work \cite{AB}, 
 sufficient conditions were obtained to guarantee the uniform persistence in systems of differential equations with delay, motivated by population models. 
To this end,   a guiding-type function was considered and, under  different assumptions, several results from the ODE theory were successfully extended to the DDE framework. 
Moreover,  the persistence result was then applied to obtain sufficient conditions for the existence of an invariant subset of the state-space containing all possible   periodic solutions. 

\medskip


\subsection{Discussion of the main results}

In this work, persistence and existence of periodic solutions of the Nicholson equation are proved applying guiding functions techniques.
This method was already explored in our previous paper \cite{AB} for abstract systems. 
Here, we investigate in detail the application of the procedure to this  specific context, with the aim of obtaining 
more explicit conditions. 
In contrast with our previous approach, we shall employ here nonsmooth functions $v(t):=\min\{x_i(t)\}$ and $u(t):=\max\{x_i(t)\}$ which, as shown below, allow to avoid the use of more  technical tools.
An additional advantage of the method is the fact that  persistence for the $N$-dimensional system can be proven  via a reduction to a one-dimensional problem. 
We shall deal with  nonlinear  harvesting and mortality terms, thus extending previous results in the literature for the $N$-dimensional case.

Moreover, it shall be  shown how persistence 
can be employed to obtain conditions for the existence of periodic solutions.
To this end, degree theory techniques  and  the Poincaré-Miranda theorem shall be applied.

\medskip



\subsection{Outline}

This paper is organized as follows. 
In the introduction, we discuss in detail the framework of the equation and some previous results that are relevant to our analysis.
We start with a  preliminary section, in which  notation and some basic definitions are introduced. 
In section \ref{persist-sec}, we analyze persistence of the Nicholson's equation in the scalar and $N$-dimensional cases.
In  section  \ref{periodic}, we prove the existence of $T$-periodic solutions of the system.
Finally, the attractiveness of the trivial equilibrium is studied in section \ref{atract}.


\section{Preliminaries}

Let us introduce some definitions and notation, following  \cite{AB} and \cite{gard}.
Let us consider the delay differential equation
\begin{equation}
\label{eq}
x'(t)=f(t,x(t),x(t-\tau_j))
\end{equation}
\noindent where $f:[0,+\infty)\times[0,+\infty)^{2N}\to \R^N$ is continuous and 
$\tau_j>0$ are the delays for $j=1,...,k$.
Throughout this paper, we shall consider only positive initial conditions for (\ref{eq}), namely functions $\varphi\in 
C([-\tau^*,0],\R^N)$  with $\varphi_i\ge 0$, where $\tau^{*}=\max\{\tau_j\}$. 
Thus, the  initial value at $t=0$ shall be written as
\begin{equation}
\label{ic}
x_0=\varphi,    
\end{equation}
where, for arbitrary $t$, the mapping $x_t\in C([-\tau^*,0],\R^N)$ is defined, as usual,  by $x_t(s):=x(t+s)$.


Since we are interested in positive solutions, we shall work on the  
positive cone of $\R^N$, that is 
$$\R^N_{+}:=\{x \in\R^N: x_i\geq 0, 1\leq i\leq N\}$$ 
equipped with  $|\cdot|$, the Euclidean norm of $\R^N$.


\medskip 

We shall say that the solutions of system (\ref{eq}) are:
\begin{itemize}
    \item \emph{strongly persistent} if 
    $$\liminf_{t\to+\infty} |x(t)|>0,\quad \forall\, x \in  \R^N_{+}\backslash\{0\}.$$
    
    \item \emph{uniformly persistent} if there exists $\varepsilon>0$ such that
    $$\liminf_{t\to+\infty} |x(t)|>\varepsilon,\quad   \forall\, x \in \R^N_{+}\backslash\{0\}.$$
\end{itemize}


In order to study periodic solutions, the following  notation shall be adapted from \cite{AD}. Let $C_T$ be the Banach space of continuous $T$-periodic vector functions, with the standard norm $\|\cdot \|_\infty$. We shall focus our attention on the set $X\subset C_T$ defined as the interior of the positive cone $\mathcal{C}$, namely
$$
X:=\{x(t)\in C_T : x_i(t)>0,{} 1\leq i\leq N\}.
$$


\section{Persistence in Nicholson's equation}

\label{persist-sec}

This section  shall be essentially based on the technique of guiding functions, which  was already used in our previous work \cite{AB}. 
A relevant fact in the present approach is that  the analysis of the system can can be reduced to a one-dimensional case. In more precise terms, our proof of strong persistence of the non-trivial solutions shall rely on the choice, as guiding function, of the function $\min\{x(t)\}$, thus allowing to apply  the results  of the case  $N=1$. 
 This is why, for a more comprehensive presentation, we shall firstly deal with the scalar Nicholson equation with nonlinear terms and give conditions for strong persistence and uniform   persistence. In a second subsection, we shall proceed with the general  $N$-dimensional system.

The main result of this section concerns the  uniform persistence of the solutions for the Nicholson's $N$-dimensional equation. The proof will be carried out after proving the  strong persistence as a first step. 


\subsection{Persistence in the scalar case} \label{persist}

We shall consider a scalar equation with  several delays and nonlinear harvesting. As an extension of previous results in this topic, we shall also consider a nonlinear mortality rate. 
Although the harvesting and mortality terms could be unified, we shall keep them separate for a better understanding of the $N$-dimensional case.

Our version of the scalar Nicholson equation reads as follows:  
\begin{equation}
\label{escalar}
x'(t)=  -d(t, x(t)) +\sum_{j=1}^{k} p_j(t) f(x_{\tau_j}(t))- H(t,x(t))
\end{equation} 
Here, the functions $d(t, x), H(t,x), p_j(t)$ are continuous and nonnegative.

Prior to our analysis, it is   worthy mentioning that the 
conditions obtained below can be regarded as natural, in the sense 
that extend previously known conditions obtained when $d$ and $H$ are linear functions.  
For the general case,  let us set $G(t,x):=d(t,x)+H(t,x)$,  assume that $\frac{G(t,x)}{x}$ is bounded from above and define
\begin{equation*}
\begin{split}
G^{0}(t) &:=\limsup_{x\to0^{+}}  \frac{d(t, x)+ H(t,x)}{x}.\\
\end{split}
\end{equation*}
Furthermore, we shall assume that the limit  is uniform in $t$. 
In particular,
this holds when $d(t,0)=H(t,0)$, provided that $d$ and $H$ are smooth functions.
\medskip

\begin{rem}
It is easy to see that if the (scalar) initial condition satisfies $\varphi>0$, then $x_t>0$ for all $t$. 
\end{rem}


\begin{rem}
Let us notice that since $\frac{G(t,x)}{x}$ is bounded from above, there exists $k_0$, such that $G(t,x(t))\le k_0 x(t)$.
This, in turn, implies that $x(t) \ge e^{-k_0\tau^*} x_{\tau_j}(t)$ for each $1\leq j\leq k$.

\noindent Indeed, since $x'(t)\geq -k_0 x(t)$, we have that $\ln x(t) - \ln x_{\tau_j}(t) \ge -k_0\tau_j$, whence  $x(t) \ge e^{-k_0\tau^*} x_{\tau_j}(t)$.
\end{rem}


\medskip

In order to guarantee strong persistence in the scalar case, the following condition  shall be assumed:
\begin{equation}\label{H0}
G^0(t)<(1-\delta) \sum_{j=1}^{k} p_j(t)
\end{equation}
for some $\delta<1$ and all $t$. 
  With this in mind, let us establish our first result:

\begin{prop}\label{prop1}
Assume (\ref{H0}) holds,
then the solutions of (\ref{escalar}) are strongly persistent.
\end{prop}

\begin{proof}
Suppose that  $x(t)$ is a non-trivial solution such that $\liminf_{t\to+\infty} x(t)= 0$
and fix $t_n\nearrow+\infty$ such that $x(t_n)=\min_{t\in[t_{n-1},t_n]}x(t)$.
This implies that $\lim_{n\to+\infty}x(t_n)=0$, $x'(t_n)\le 0$ and  $x(t_n)\leq x_{\tau_j}(t_n)$ for $n$ large enough. 
In turn, we get that $f(x(t_n))\le f(x_{\tau_j}(t_n))$ because,  for $x(t_n)\le e^{-k_0\tau^*}$ it is seen  that $x_{\tau_j}(t_n)\le 1$.  
Thus,
$$
\begin{array}{cc}
 \sum_{j=1}^{k} p_j(t_n) f(x_{\tau_j}(t_n))
\geq 
\sum_{j=1}^{k} p_j(t_n) f(x(t_n))\,\hbox{ when $n\gg0$}. 
\end{array}
$$
From (\ref{H0}), we deduce that 
$$
\frac{G(t_n, x(t_n))}{x(t_n)} <(1-\delta)\sum_{j=1}^{k} p_j(t_n)\,\hbox{ for $n$ large.}
$$
Moreover, we can find $n_0$ such that $e^{-x(t_n)}>1-\delta$ for $n\geq n_0$; thus,  the following contradiction yields
$$ 
0\geq x'(t_n)\geq x(t_n)\left[\sum_{j=1}^{k} p_{j}(t_n)e^{-x(t_n)} -\frac{G(t_n,x(t_n))}{x(t_n)}\right]>0
$$
\noindent provided that $x(t_n)< \min\{e^{-k_0\tau^*}, -\ln(1-\delta)\}$.
\end{proof}

\medskip


However, notice that the previous condition is not enough to guarantee a uniform bound such that  all the solutions  
stay asymptotically away from zero.  
Specifically, it might be the case that the distance between $x(t)$ and $x(t-\tau_j)$ remains large.
This problem is avoided by  the next assumption:

\medskip 
There exists $c>0$ such that 
\begin{equation}\label{p-bounded inf}
    \sum_{j=1}^{k} p_{j}(t)\ge c
\end{equation}
for all $t$. 
\medskip
 It is readily verified that, together with (\ref{H0}), condition  
 (\ref{p-bounded inf}) implies the existence of $r_0>0$  
such that, for each $r\in (0,r_0)$,
\begin{equation}\label{H1_uni}
\liminf_{t\to+\infty, x, y\to r^-} \sum_{j=1}^{k} p_{j}(t) f(y) -G(t,x)>0.
\end{equation}   

\medskip

\begin{thm}\label{teo1}
Assume that (\ref{H0}) and (\ref{p-bounded inf}) hold,
then the solutions of (\ref{escalar}) are uniformly persistent.
\end{thm}

\begin{proof}
Suppose that $x(t)$ is a solution such that $\liminf_{t\to\infty}x(t)=r\in (0,r_0)$.  
There are 3 cases we need to analyze:
\begin{enumerate}
\item \label{caso1}
The solution $x(t)\ge r$ for large values of $t$. 
\item\label{caso2}  
The solution verifies $\lim_{t\to+\infty} x(t)=r^-$.
\item \label{caso3}
The solution $x(t)$ oscillates around $r$. 
\end{enumerate}
    
For the first case, let us set $t_n\to+\infty$ such that $x(t_n)=\min_{t\in[t_{n-1},t_n]}x(t)$. 
Hence $x'(t_n)\le 0$ and $x(t_n) \le x_{\tau_j}(t_n)$.
Since $\lim_{t\to+\infty}x(t_n)=r$, there is $n\gg0$ such that $x(t_n)<r_0$ and,
as in the previous proof, the latter implies that $x_{\tau_j}(t_n)\leq 1$.
Fix $r_0< e^{-k_0\tau^*}$ in order to apply $(\ref{H0})$; then   $x(t_n)<r_0$ for $n$ large. As before, there exists $\delta<1$ such that
$$
\frac{G(t_n, x(t_n))}{x(t_n)}<(1-\delta)\sum_{j=1}^{k} p_j(t_n);
$$
therefore, 
$$0\ge \sum_{j=1}^{k} p_j(t)e^{-x(t_n)}  - \frac{G(t,x(t_n))}{x(t_n)}>0,$$
a contradiction.
\medskip

Let us analyze the other two cases:
If  $x(t)\to r^-$, then also $x_{\tau_j}(t)\to r^-$.  
Because of  (\ref{H1_uni}) we can choose a constant $c_0>0$  such that
$$ 
\liminf_{t\to+\infty, x,y\to r^-}
\sum_{j=1}^{k} p_j(t)f(y) -G(t,x)>c_0.
$$ 
It follows that $x'(t)>{c_0}$ for $t\gg 0$, which cannot happen. 
Finally, in the third case, we can fix a sequence $s_n\to +\infty$  such that $x(s_n)>r$ and $x(s_n)\to r^+$.  
Let us consider $t_n \in [s_{n-1},s_n]$ such that $x(t_n)= \min_{t\in [s_{n-1},s_n]} x(t)$.
Thus $x(t_n)\to r$, with $x'(t_n)\le 0$ and $x_{\tau_j}(t_n)\ge x(t_n)$ for $n$ sufficiently large. 
Hence, the situation is similar to the first case, which yields a contradiction.
\end{proof}


\subsection{Persistence in the $N$-dimensional case}

As mentioned in the introduction, our study is devoted to a general $N$-system of Nicholson's equation. This section shall be based on the previous scalar analysis to deduce   persistence results when $N>1$. It is observed  that the function $\min\{x(t)\}$  will be crucial to prove the results.

For the reader's convenience, let us write again  the system introduced in the introduction, namely
\begin{equation}\label{nich-sys-gral-2}
x_i'(t)=-d_i(t, x_i) + \sum_{l=1, l\neq i}^{N} b_{i,l}(t, x_l)+ \sum_{j=1}^{k} p_{i,j}(t) f(x_{i,\tau_j}(t)) -H_i(t,x)
\end{equation}
\noindent for $i=1,...,N$.

The functions $b_{i,l}(t, x_l), d_i(t, x_i), H_i(t,x), p_{i,j}(t)$ are nonnegative and possibly nonlinear. 
Let us set $G_i(t,x)=d_i(t,x)+ H_i(t,x)$ and  assume as before that $\frac{G_i(t,x)}{x_i}$ are bounded from above. 
\medskip 

We shall show strong and uniform persistence under appropriate conditions, which are analogous to those obtained in the scalar case. 
Notice that each equation has an extra sum of terms $b_{i,l}$, which it is interpreted as the mutual cooperation interactions between the populations. 
\medskip

\begin{rem}
As in the scalar case, we assume the existence of  $k_0$ such that $G_i(t,x(t))\le k_0 x_i(t)$ for $1\leq i\leq N$.
It follows that $x_i(t) \ge e^{-k_0\tau^*} x_{i,\tau_j}(t)$ for $1\leq j\leq k$.
\end{rem}


For $i, l=1, ...N$, set
\begin{equation*}
\begin{split}
G^{0}_{i}(t) &:=\limsup_{x_i\to0^{+}}  \frac{G_i(t, x)}{x_i}.\\
b_{0_{i,l}} (t) &:=\liminf_{x_l\to0^{+}}  \frac{b_{i,l}(t, x_l
)}{x_l}, \hbox{ for }l\ne i.
\end{split}
\end{equation*}

\noindent It is assumed that all of the limits are uniform on $t$ and $G^{0}_i$ it is also uniform on $x_l$ with $l\neq i$.
\medskip

\begin{prop}\label{prop2}
The solutions of the system (\ref{nich-sys-gral-2}) are strongly persistent if, for some $\delta>0$,  
\begin{equation}\label{H0_syst}
 G^0_{i}(t)<(1-\delta)\left(\sum_{l=1, l\neq i}^{N} b_{0_{i,l}} (t)+ \sum_{j=1}^{k} p_{i,j}(t)\right) 
\end{equation}
\end{prop}

\begin{proof} 
Suppose $x(t)$ is a solution such that $\liminf_{t\to+\infty}|x(t)|=0$
and  choose a sequence $s_n\nearrow+\infty$ such that $|x(s_n)| \le |x(t)|$ for $t\le s_n$ and
$\lim_{n\to+\infty}|x(s_n)|=0$. 
Thus, $\lim_{n\to+\infty}x_{i}(s_n)=0$ for $i=1,...,N$.
Let us set $v(t):=\min\{x_i(t)\}$ and take $t_n$ such that $v(t_n)=\min_{t\in[s_0,s_n]} v(t)$.
Passing to a  sub-sequence, we can find $i_0\in\{1,...,N\}$ such that $x_{i_0}(t_n)=v(t_n)$ for all $n$. 
Without loss of generality, we can assume that 
$x_1(t_n)=v(t_n)$.
This implies that $x_{1}'(t_n)\leq0$ and  $x_1(t_n)\le x_{1,\tau_j}(t_n)$ for $n$ large.
For simplicity, let us use the notation $x_i(t_n)=x_i$.
Following the reasoning from the scalar case, we know that for $n$ large

$$
x_1'\geq 
x_1\left[
\sum_{l=2}^N \frac{b_{1,l}(t_n,x_l)}{x_1} 
+\sum_{j=1}^{k}p_{1,j}(t_n)e^{-x_1} -\frac{G_1(t_n,x)}{x_1}\right],    
$$

\noindent and since $x_1\leq x_l$ for all $l\in \{2,...,N\}$, 

$$
x_1'\geq x_1\left[\sum_{l=2}^N \frac{b_{1, l}(t_n, x_l)}{x_{l}}+ \sum_{j=1}^{k} p_{1,j}(t_n)e^{-x_1} -\frac{G_1(t_n,x)}{x_1}\right]. 
$$
From  (\ref{H0_syst}) it holds, for $n$ large: 
$$
\frac{G_1(t_n, x)}{x_1}<\sum_{l=2}^{N}\frac{b_{1,l}(t_n, x_l
)}{x_l}+(1-\delta) \sum_{j=1}^{k} p_{1,j}(t_n).  
$$
Moreover, there exists $n_0$ such that $e^{-x_1(t_n)}>1-\delta$, for $n\geq n_0$, so

$$
0\geq\sum_{l=2}^N \frac{b_{1, l}(t_n, x_l)}{x_{l}}+ \sum_{j=1}^{k} p_{1,j}(t_n)e^{-x_1} -\frac{G_1(t_n,x)}{x_{1}}>0
$$
\noindent we get a contradiction.
\end{proof}

\medskip


The proof of uniform persistence in the $N$-dimensional case is a bit more technical. With the scalar case in mind, we consider the assumption:

\medskip 
There exists $c>0$ such that 
\begin{equation}\label{p-bounded-sys}
    \sum_{j=1}^{k} p_{i,j}(t)\ge c
\end{equation}
for all $t$. 
As before, (\ref{H0_syst}) and (\ref{p-bounded-sys}) imply the existence of $r_0>0$ such that,
for each $r\in (0,r_0)$,
\begin{equation}\label{H1_syst_uni}  
\liminf_{t\to+\infty\,, x_i,\, y_i\to r^-} \sum_{j=1}^{k} p_j(t)f(y_i) + \sum_{l=1, l\neq i}^{N} b_{i,l}(t, x_l) - G_i(t,x)>0.\,
\end{equation}   

\medskip

Finally, we state the main result of this section:

\begin{thm}\label{prop4}
Assume (\ref{H0_syst}) - (\ref{p-bounded-sys}) hold,
then the solutions of (\ref{nich-sys-gral-2}) are uniformly persistent.
\end{thm}

\begin{proof}
Let us begin by setting $\varepsilon_0$ small enough to use condition (\ref{H0_syst}) and
let us suppose that $x(t)$ is a solution such that $\liminf_{t\to\infty}|x(t)|=\varepsilon_0$.
It is clear to see that $\liminf_{t\to+\infty}x_i(t)=r_i\leq \varepsilon_0$ for all $i$.
As in the scalar case, three cases  have to be considered.
Let us set $\tilde{r}:=\liminf_{t\to+\infty}v(t)$, where $v(t)=\min\{x_i (t)\}$.  
First, let us assume that $\tilde{r}=0$.
In this case, we can set $t_n\to+\infty$ such that $\lim_{n\to+\infty}v(t_n)=0$ and $v(t_n)=\min_{t\in[t_0,t_n]} v(t)$.
Clearly, we can find $i_0$ such that $x_{i_0}(t_n)=v(t_n)$ for all $n$. 
Thus, following an analogous reasoning of the proof of Proposition (\ref{prop2}) we get a contradiction.


Let us 
examine the case when $v(t)\geq\tilde{r}$ for $t\gg0$. 
We may choose $t_n\to+\infty$ such that 
$v(t_n)=\min_{t\in[t_0,t_n]} v(t)$ and $i_0$ such that $v(t_n)=x_{i_0}(t_n)$ for all $n$.
Then, $x_{i_0}'(t_n)\leq 0$ and $x_{i_0,\tau_j}(t_n)\geq x_{i_0}(t_n)$ for $n\gg0$.
Applying the same procedure as before, we get
$$
x_{i_0}'\geq 
x_{i_0}\left[\sum_{l=1 l\neq i_0}^N \frac{b_{i_0,l}(t_n,x_l)}{x_{i_0}} 
+\sum_{j=1}^{k}p_{i_0,j}(t_n)e^{-x_{i_0}} -\frac{G_{i_0}(t_n,x)}{x_{i_0}}\right],    
$$
Since $x_{i_0}\leq x_l$, 
$$
x_{i_0}'\geq 
x_{i_0}\left[
\sum_{l=1 l\neq i_0}^N \frac{b_{i_0,l}(t_n,x_l)}{x_{l}} 
+\sum_{j=1}^{k}p_{i_0,j}(t_n)e^{-x_{i_0}} -\frac{G_{i_0}(t_n,x)}{x_{i_0}}\right].    
$$
Since $\varepsilon_0$ has been chosen such that each $x_i(t)$ is sufficiently close to zero for $n\gg 0$,   a contradiction follows.
If  $v(t)\to \tilde{r}^-$, then also $v_{\tau_j}(t)\to \tilde{r}^-$.  
Without loss of generality, we can assume that $\tilde{r}<r_0$.
As before, via sub-sequences we can choose $i_0$ such that $v(t_n)=x_{i_0}(t_n)$, so in the presence of (\ref{H1_syst_uni}) we can pick a constant $c_0>0$ such that for large values of $n$
$$  
\sum_{j=1}^{k} p_j(t_n)f(x_{i_0}) + \sum_{l=1, l\neq i_0}^{N} b_{i_0,l}(t_n, x_l) - G_{i_0}(t_n,x)>c_0.
$$   

It follows that $v'(t_n)>c_0$ for $n\gg 0$, which cannot happen. 
In the remaining case, we can fix a sequence $s_n\to +\infty$  such that $v(s_n)>\tilde{r}$ and $v(s_n)\to \tilde{r}^+$.  
Let us consider $t_n \in [s_{n-1},s_n]$ such that $v(t_n)= \min_{t\in [s_{n-1},s_n]} v(t)$.
Thus $v(t_n)\to \tilde{r}$, with $v'(t_n)\le 0$ and $v_{\tau_j}(t_n)\ge v(t_n)$ for $n$ sufficiently large. 
Hence, the situation is again as in the  first case and the contradiction follows.
\end{proof}




\section{Periodic solutions in the $N$-dimensional case}
\label{periodic}

In this section, we investigate the existence of $T$-periodic solutions in the context of our generalization of the $N$-dimensional Nicholson equation.
For this purpose, we shall use a topological and degree theory methods. 
In more precise terms, we will define a suitable operator on a proper set of the space of continuous $T$-periodic functions and show that the fixed points of the operator are the periodic solutions of our system.
The techniques of the Brouwer degree and the Leray-Schauder degree are used to prove the existence of fixed points and, consequently, of periodic solutions.
\medskip


Let us introduced the standard continuation method \cite{maw}, following the notation of \cite{AD2}.  
{Recall that $X\subset C_T$ is defined as the interior of the positive cone of $C_T$ and define the operators}
$$\Phi:=(\phi_1,...,\phi_N):X\to C_T,\qquad  \mathcal{K}:X\cap C^1(\R,\R^{N})\to C_T,$$
\noindent in the following way:
\begin{equation*}
\begin{split}
\phi_i(x)(t)&:=-d_i(t, x_i) + \sum_{l=1, l\neq i}^{N} b_{i,l}(t, x_l)+ \sum_{j=1}^{k} p_{i,j}(t) f(x_{i, \tau_j}(t)) -H_i(t,x),\\
\mathcal{K}\varphi:=x&\,\hbox{ is the unique solution of the equation $x'(t)=\Phi(t)$ such that $\overline{x}=0$}.
\end{split}    
\end{equation*}

Let us identify the positive constant functions of $X$ with ${(0,+\infty)}^N$ and define $g:{(0,+\infty)}^N\to\R^N$ as

$$g(x)=-\overline{\Phi(x)}=-\ddfrac{1}{T}\int_{0}^{T}\Phi(x)(t)\, dt.$$

In order to state the lemma, let us define
\begin{equation*}
\begin{split}
G_{i, \infty}(t)&:=\liminf_{x_i\to+\infty}  \frac{G_i(t, x)}{x_i},\\
b^{\infty}_{i,l}(t)&:=\limsup_{x_l\to+\infty}  \frac{b_i(t, x_l
)}{x_l},\hbox{ $l\neq i$}.\\
\end{split}
\end{equation*}

\noindent Again, we shall assume that the limits are uniform in $t$ and $G_{i, \infty}$ it is uniform in $x_l$ with $l\neq i$.  
We consider the following hypothesis: 

\begin{equation}\label{per}
    G_{i, \infty}(t)>\sum_{l=1, l\neq i}^{N} b^{\infty}_{i,l}(t)+\beta,
\end{equation}   
 for some positive $\beta$.

\begin{lem}\label{lema1}
Assume (\ref{H0_syst}) and (\ref{per}) hold 
then the solutions of the system $x'=\lambda\Phi(x)$ are uniformly bounded for every $\lambda\in[0,1]$.
\end{lem}
\medskip

\begin{proof}
Let us a consider $x=(x_1,...,x_n)$ a solution of the system $x'=\lambda\Phi(x)$ and set $R:=\max_{1\le i\le N}\{x_i(t), t\in[0,T]\}$.
As before, we may assume that $R=\vert\vert x_1 \vert\vert_\infty$ and fix $t^*\in[0, T ]$ such that $x_1(t^*)=R$. From the first equation of the system $x'=\lambda\Phi(x)$, it follows that $x_1'(t^*) = 0$, so 

$$
G_1(t^*,x(t^*)) = \sum_{l=2}^{N} b_{1,l}(t^*, x_l(t^*))+ \sum_{j=1}^{k} p_{1,j}(t^*)f(x_{1,\tau_j}(t^*)).
$$
\medskip

\noindent Since $f(x_{1,\tau_j}(t^*))\leq e^{-1}$ 
we deduce
\begin{equation*}
\begin{split}
R\left( \frac{G_1(t^*,x(t^*))}{R}\right)
\leq R\left(\sum_{l=2}^{N}\frac{b_{1,l}(t^*, x_l(t^*))}{x_l(t^*)}\right)+\frac{p^*}{e}\quad\text{where $p^*=\max_{1\le i\le N}\sum_{j=1}^{k} p_{i,j}$.}
\end{split}    
\end{equation*}
Thus,  
$$
R\left(\frac{G_1(t^*,x(t^*))}{R}-\sum_{l=2}^{N}\frac{b_{1,l}(t^*, x_l(t^*))}{x_l(t^*)}\right)
\leq
\frac{p^*}{e}.
$$
From (\ref{per}), there exist a constant $\tilde{R}$ such that for $x_i\geq\tilde{R}$
\begin{equation}
\frac{G_1(t, x)}{x_1}>\sum_{l=2}^{N}\frac{b_{1,l}(t, x_l)}{x_l} +\beta
\end{equation}
In particular, 

$$
\frac{G_1(t^*,x(t^*))}{R}>\sum_{l=2}^{N}\frac{b_{1,l}(t^*, x_l(t^*)
)}{x_l(t^*)} +\beta
.$$ 
Thus, $x_i(t)\leq\max\{\tilde{R}, \frac{p^*}{e\beta}\}:=R_0$
for all $t$ and $i = 1,..., N$.

Similarly, set $\varepsilon:=\min_{1\le i\le N}\{x_i(t), t\in[0,T]\}$. We may assume that the minimum occurs at the first coordinate 
and take $t_*\in[0, T]$ such that $x_1(t_*)=\varepsilon$. 
As before,  
 
$$
G_1(t_*,x(t_*)) = \sum_{l=2}^{N} b_{1,l}(t_*, x_l(t_*))+ \sum_{j=1}^{k} p_{1,j}(t_*)f(x_{1,\tau_j}(t_*)).
$$

Notice that if $R_0<\varepsilon$ we obtain the desired bounded. 
Without loss of generality, let us assume that $R_0\geq 1$ and choose
$R_1\in(\varepsilon,1]$ such that $f(R_1)=f(R_0)$. Thus, since $f$ has a global maximum at 1, either  $f(x_{1,\tau_j}(t_*))\geq f(R_0)=f(R_1)\geq f(\varepsilon)$ or directly $f(x_{1,\tau_j}(t_*))\geq f(\varepsilon)$.
In any case, we get

\begin{equation*}
\begin{split}
\frac{G_1(t_*,x(t_*))}{\varepsilon}- \sum_{l=2}^{N}\frac{b_{1,l}(t_*,x_l(t_*))}{x_l(t_*)}
&\geq\sum_{j=1}^{k} p_{1,j}(t_*)e^{-\varepsilon}\\
&\geq\sum_{j=1}^{k} p_{1,j}(t_*)+p^*(e^{-\varepsilon}-1)
\end{split}    
\end{equation*}
Using (\ref{H0_syst}), let us choose $\tilde{\varepsilon}>x_i$ and $\gamma<p^*$  such that 
$$
 \frac{G_1(t,x)}{x_1}-\sum_{l=2}^{N}\frac{b_{1,l}(t, x_l)}{x_l}<\sum_{j=1}^{k} p_{1,j}(t)-\gamma. 
$$
In particular,

$$
 \frac{G_1(t_*,x(t_*))}{x_1(t_*)}-\sum_{l=2}^{N}\frac{b_{1,l}(t_*, x_l(t_*)
)}{x_l(t_*)}<\sum_{j=1}^{k} p_{1,j}(t_*)-\gamma.  
$$

Finally, 
$$x_i\geq\varepsilon_0:=\min\{\tilde{\varepsilon}, -\ln(1-\frac{\gamma}{p^*})\}.$$
\end{proof}


We are in conditions to state the main result of this section:  

\begin{thm}\label{teo-per}
Assume the hypotheses of the previous lemma hold, then the system (\ref{nich-sys-gral-2}) has at least one non-trivial $T$-periodic solution.
\end{thm}

\begin{proof}
Let us define the homotopy, 
$$
F_{\lambda}(x)=x-\overline{x}-\overline{\Phi(x)}-\lambda\mathcal{K}(\Phi(x)-\overline{\Phi(x)}),\quad\lambda\in[0,1]
$$
and note that $x$ is a solution of $x'=\lambda\Phi(x)$ if and only if $x$ is a zero of $F_{\lambda}$. 
Let us set $\Omega=\{x\in X: \varepsilon\leq x_i(t)\leq R\}$, $\varepsilon<\varepsilon_0$ and $R>R_0$ from the previous lemma.
We will prove that $F_{\lambda}$ is non-zero at the boundary of $\Omega\times[0,1]$ and that the degree of $F_0$ on $0$ is non-zero. 
From the degree invariance, this will imply that $\deg(F_{1})\neq 0$ in $\Omega$, hence the existence of a non-trivial $T$-periodic solution of the system.

To this end, 
it suffices to show that $F_0$ restricted to the subspace of constant functions is homotopic to a translation of the identity of $\R^N$, which has a (unique) zero in $U:=\Omega\cap\R^N$. 
It is easy to verify that $F_0$, when evaluated in  constant functions, is exactly equal to $g$.
Therefore, it is enough to show that
\begin{equation*}
\left\{
\begin{aligned}
g_1(\varepsilon, x_2,..., x_N)&<0,\\ 
g_2(x_1, \varepsilon,..., x_N)&<0,\\
&\vdots\\
g_N(x_1, x_2,..., \varepsilon)&<0.\\
\end{aligned}
\right.
\end{equation*}

\begin{equation*}
\left\{
\begin{aligned}
g_1(R, x_2,..., x_N)&>0,\\ 
g_2(x_1, R,..., x_N)&>0,\\
&\vdots\\
g_N(x_1, x_2,..., R)&>0.\\  
\end{aligned}
\right.
\end{equation*}
This implies that the linear homotopy does not vanish in $\partial (U\times[0,1])$, hence that $\deg(g, U, 0)=1$. 
\medskip 

Let us firstly suppose that $x_1=\varepsilon\leq x_i\leq R$, $i=2,..., N$, then
  
\medskip 

\noindent$
g_1(\varepsilon, x_2,..., x_N)=$
$$=\dfrac{1}{T}\int_{0}^{T}\left(G_1(t,\varepsilon, x_2,...,x_N) - \sum_{l=2}^{N} b_{1,l}(t,x_l)- \sum_{j=1}^{k} p_{1,j}(t) f(\varepsilon)
\right)\,dt.
$$
{Thanks to (\ref{H0_syst}), we may choose $\varepsilon>0$ such that $e^{-\varepsilon}>1-\delta$ and}

$$
\dfrac{G_1(t,\varepsilon, x_2,...,x_N)}{\varepsilon} < \sum_{l=2}^{N} \dfrac{b_{1,l}(t,x_l)}{\varepsilon}+ e^{-\varepsilon}\sum_{j=1}^{k} p_{1,j}(t).
$$
{It follows that   $g_1(\varepsilon, x_2,..., x_N)<0$.}
It is clear that this can be done for each $g_i$, namely $g_i(x_1,..., x_{i-1}, \varepsilon, x_{i+1},..., x_N)<0$, $2\leq i\leq N$.

\medskip

Next, consider the case  $\varepsilon\leq x_i\leq R=x_1$. We have:
  
\medskip 
\noindent
$g_1(R, x_2,\ldots , x_N)  =$
$$=
\dfrac{1}{T}\int_{0}^{T}\left(G_1(t,R, x_2,...,x_N) - \sum_{l=2}^{N} b_{1,l}(t,x_l)- \sum_{j=1}^{k} p_{1,j}(t) f(R)
\right)\,dt.
$$ 

\noindent{From (\ref{per}), it is deduced, for $R\gg0$,}
$$
\frac{G_1(t, R, x_2,..., x_N)}{R}>\sum_{l=2}^{N}\frac{b_1(t, x_l
)}{R} +\beta.
$$ 

\noindent
{Therefore, choosing $\dfrac{\beta}{p^*}>e^{-R}$ it follows that  $g_1(R, x_2,..., x_N)>0$.
An analogous result holds for each $g_i(x_1,...x_{i-1}, R,..., x_N)$, $2\leq i\leq N$.}
\medskip

Finally, 
$$\deg(F_1, \Omega, 0)=\deg(F_0, \Omega, 0)=\deg(g, U, 0)=1$$
and so completes the proof. 
\end{proof}

\medskip

\section{Zero global attractor}

\label{atract}
The goal of  this section is to prove   that, under appropriate conditions, zero is a global attractor for the positive trajectories. 
To this end, we shall use another suitable guiding (Lyapunov-like) function. 
\medskip


\begin{thm}
Assume there exists $c>0$ such that $\sum_{j=1}^{k} p_{i,j}(t)\ge c> 0$  and  
\begin{equation}\label{atrae}
\frac {G_{i}(t,x) }{x_i}\geq \sum_{l=1, l\neq i}^{N}\frac{{b_{i}(t, x_l)}}{x_l} + \sum_{j=1}^{k} p_{i,j}(t),\quad\hbox{$i=1,...,N$}
\end{equation}   
for all $t$ hold, then every positive solution of  $(\ref{nich-sys-gral-2})$ with initial data $\varphi>0$,  tends to $0$ as $t\to +\infty$.  

\end{thm}

\medskip

\begin{proof}
Let $x$ be a positive solution and define $u(t):=\max \{x_i(t)\}$.
Notice that, in order to prove our result, it is enough to show that $u\to 0$.
We claim that if $u\searrow \alpha$ then $\alpha=0$.
Firstly, let us suppose that $0<\alpha<1$. In this case, it is clear that since  $u_{\tau_j}(t)<1$ for $t\gg0$ and $x_{i,\tau_j}(t)\le u_{\tau_j}(t)$ we get that 

$$\frac {x_{i,\tau_j}(t)e^{-x_{i,\tau_j}(t)}}{u(t)}\le \frac {u_{\tau_j}(t)e^{-u_{\tau_j}(t)}}{u(t)}.$$

\noindent This implies, $$\lim_{t\to+\infty}\frac{f(x_{i,\tau_j}(t))}{u(t)}=e^{-\alpha}.$$

To analyze the case when $\alpha\ge 1$, let us choose $\varepsilon>0$ such that
$\alpha\le u(t)\le \alpha+\varepsilon$ for $t\gg0$. Set $\gamma< 1$ such that $\gamma e^{-\gamma}:= (\alpha+\varepsilon)e^{-(\alpha+\varepsilon)}$. If $x_{i,\tau_j}(t)\le \gamma$ we obtain

$$\frac {x_{i,\tau_j}(t)e^{-x_{i,\tau_j}(t)}}{u(t)}\le \frac {u(t)e^{-u(t)}}{u(t)}\le e^{-\alpha}.
$$

\noindent If $x_{i,\tau_j}(t)> \gamma$, $e^{-x_{i,\tau_j}(t)}\le e^{-\gamma}$ so

$$\frac {x_{i,\tau_j}(t)e^{-x_{i,\tau_j}(t)}}{u(t)}\le \frac {x_{i,\tau_j}(t)e^{-\gamma}}{u(t)}\le  
\frac {u_{\tau_j}(t)e^{-\gamma}}{u(t)}\to e^{-\gamma}.$$
Let us set $\varepsilon_0>0$ such that $\max\{e^{-\alpha}, e^{-\gamma}\}\leq (1-\varepsilon_0)$, then

$$u'(t)\leq u(t)\left(\sum_{i=1,i\ne l}^N \frac{b_i(t,x_l)}{u(t)} + (1-\varepsilon_0)\sum_{j=1}^{k} p_{i,j}(t) -\frac{G_i(t,x)}{u(t)}\right).
$$
Taking into account (\ref{atrae}), 
we obtain with $u'(t) \le c_0 u(t)$ with $c_0<0$.
In other words, 
$\ln{u(t)}'\le c_0$,
so  
$$
\int_{T+\tau*}^t \ln{u(t)}'\,dt\le \int_{T+\tau*}^t c_0 \,dt,
$$

\noindent and therefore 
$$\lim_{t\to+\infty} \ln u(t)\le \lim_{t\to+\infty} [ c_0 (t-T-\tau*) + \ln u(T+\tau*)]=-\infty.$$ 


\medskip

In order to complete the proof, we need to prove that $u$ converges to some $\alpha$.
If $u'(t_0)\geq 0$ for some $t_0$, we obtain
$u(t_0)\leq {e}^{-1}$. Indeed, since $u_{\tau_j}(t_0)e^{-u_{\tau_j}(t_0)}\leq e^{-1}$,

$$0\leq  u(t_0)\left(\sum_{i=1,i\ne l}^N \frac{b_i(t_0, x_l)}{u(t_0)} + \frac{\sum_{j=1}^{k} p_{i,j}(t_0)}{e\, u(t_0)} -\frac{G_i(t_0,x(t_0))}{ u(t_0)}\right), 
$$
whence
$$\frac{G_i(t_0,x(t_0))}{u(t_0)}\leq  \sum_{i=1,i\ne l}^N \frac{b_i(t_0, x_l)}{u(t_0)}+ \frac{\sum_{j=1}^{k} p_{i,j}(t_0)}{e\, u(t_0)}.
$$
Thus from (\ref{atrae}),
we obtain $u(t_0)\leq {e}^{-1}$. 
Recall that $\tau^{*}=\max_{j}\{\tau_j\}$. Then  $u_{\tau_j}(t)e^{-u_{\tau_j}(t)}\leq f({e}^{-1})\leq {e}^{-1}$ for $t\geq t_1+\tau^{*}$.
Again,
if $u'(t_2)\ge 0$ for some $t_2\geq t_1+\tau^{*}$, we obtain $u(t_2)\leq f({e}^{-1})$.  
Repeating the previous argument, we may choose a sequence $u_n:=u(t_n)\leq f^{n-1}({e}^{-1})$ for $t_n\geq t_{n-1}+\tau*$, so we get $$\lim_{n\to+\infty} u_n= \lim_{n\to+\infty} f^{n-1}({e}^{-1})=0.$$

\end{proof}


\section*{Acknowledgements}

\noindent This work was  supported by the projects CONICET [11220130100006CO], UBACyT [20020160100002BA] and TOMENADE [MATH-AmSud, 21-MATH-08].


\end{document}